\RequirePackage{fix-cm}
\documentclass[smallextended]{svjour3}       
\smartqed  
\usepackage{graphicx}
\usepackage{amsmath,amsfonts,latexsym}
\usepackage{amssymb}
\usepackage{subfigure}

\begin{document}
%
\title{Volume of sublevel sets versus area of level sets via Gelfand-Leray form
}

\titlerunning{Volume of sublevel sets versus area of level sets}        

\author{Trinh Duc Tai
}


\institute{Trinh Duc Tai \at
              Department of Mathematics, University of Dalat,\\
              Tel.: +84 263 3822246\\
              Fax: +84 263 3823380\\
              \email{taitd@dlu.edu.vn}           
}

\date{Received: date / Accepted: date}

%
%



\maketitle

\begin{abstract}
In this paper we give a relation between the volume of sublevel sets and the area of level sets using a Gelfand-Leray form. As a consequence, we give an estimation of the volume of sublevel sets. In particular we give a proof of the known fact that the derivative of the volume of a $n$-dimensional ball with respect to the radius equals the area of the sphere which bounds the $n$-dimensional ball.
\end{abstract}

%

\section{Introduction}
For a geometric figure in $3$-dimensional space, there are some intimate relationships between its volume and surface area. For example,  the variation of surface area of a figure under a dilation implies the variation of its volume, and conversely. It is then a natural question to ask how the  volume of an object depends on its surface area. 

Starting from a slightly different approach to this problem, our
goal is to explore the relation between
the volume $V(t)$ of sublevel sets and the area $A(t)$ of level sets of a certain function (detailed definitions of these quantities are given in Section 2). In fact, this investigation stems from  current studies of oscillatory integrals of the following form 
\begin{equation}
I(\lambda) = \int_{\mathbb{R}^n}e^{i\lambda f(x)}g(x)dx
\label{tpdd}
\end{equation}

Many of recent studies of estimates for oscillatory integrals lead to estimating volumes of sublevel sets of lower-bounded functions in $\mathbb{R}^n$. For $n=1$, one of results is well known as van der Corput's lemma (see \cite{stein}, page $332$, Prop. $2$). 

Our main goal is to establish an explicit relation between $V(t)$ and $A(t)$ for a given sufficiently smooth function $f$. More concretely, we show that (for $t\in \mathbb{R}$)
\begin{equation}
V'(t)=\frac{A(t)}{\|\nabla f(\xi)\|},
\label{main}
\end{equation}
for some $\xi \in f^{-1}(t)$.

 We can derive from this formulation many typical events, one of which is that the derivative of a $n-$dimensional ball with respect to the radius equals the area of its boundary. Thanks to the gradient inequality, we can obtain an estimate of the decay rate of $V(t)$ when $t\longrightarrow t_0$, where $t_0$ is a critical value of $f$.

This paper is organized as follows: In Section $2$ we recall some  primary facts on Gelfand-Leray form. The main theorem and its consequence are given in Section 3. A slight generalization for the case where $f$ is piece-wise smooth is mentioned in Section 4.

\section{Gelfand-Leray form }\label{}
The residue theory on complex manifolds has been intensively studied in a remarkable work of  Leray \cite{Leray}, in which the general residue forms have been defined. In the real case, others have introduced Gelfand-Leray form (cf. \cite{Arn}, Chapter $7$, and references therein) in an attempt to explore the asymptotic behaviors of oscillatory integrals whose phase functions have non-isolated singular points.

In this section we are going to recall some elementary facts about Gelfand-Leray form:
\begin{proposition}\label{md1}
Let $U\subset \mathbb{R}^n$ be a non-empty open subset of $\mathbb{R}^n$, and $\omega$ be a  differential $k$-form on $U$, where $k\ge 1$. If $\alpha$ is a non-zero differential $1$-form on $U$ such that
\begin{equation}\label{dk1}
\alpha \wedge \omega =0 
\end{equation}
then, there exists a unique differential $(k-1)$-form $\psi$ satisfying
\begin{equation}\label{}
\omega = \alpha \wedge \psi\,.
\end{equation}
\end{proposition} 
The proof of Proposition~\ref{md1} can be found in many text books on differentiable manifolds (e.g. \cite{Lor}, Chapter $12$). The essential feature of the proof is just choosing a basis of the $k$-forms $\Omega^k(U)$ on $U$ such that $\alpha$ is a factor of $\omega$ when there are written in that basis. Then, $\psi$ is given by the remaining factor. 

 From now, assume that $\omega = g(x)dx_1\wedge\cdots\wedge dx_n$ is a given smooth differential $n-$form on $U$. We consider a smooth function $f: U\subset\mathbb{R}^n\to \mathbb{R}$ such that $df(x_0)\ne 0$, for some $x_0\in U$. Then, the condition \eqref{dk1} is satisfied in some neighborhood of $x_0$ with the $1$-form $\alpha = df$. As a consequence of Proposition~\ref{md1}, we have
\begin{proposition}\label{sutontai}
We use the notations of Proposition~\ref{md1}. Consider a smooth form $\omega$ defined on $U$ by $gdx_1\wedge\ldots\wedge dx_n$, where $g$ is a smooth function defined on $U$. Let $f:U\subset\mathbb{R}^n\rightarrow \mathbb{R}$ be a smooth function defined on $U$. Assume that the differential $df(x_0)$ of $f$ at a point $x_0\in U$ does not vanish. Then,  there is a neighborhood $U_{x_0}$ and a smooth form $\psi$ on $U_{x_0}$, such that in $U_{x_0}$, we have: 
\begin{equation}\label{dn1}
\omega = df \wedge \psi\,.
\end{equation}

Furthermore, the restriction of $\psi$ on each fiber $f^{-1}(t)\cap U_{x_0}$ is uniquely defined  for $t$ sufficiently close to $t_0 = f(x_0)$.
\end{proposition}
\begin{definition} The differential form $\psi$ in Proposition~\ref{sutontai} is called a \textit{Gelfand-Leray form}.
\end{definition}
To specify the dependence on $\omega$ and $f$, we always use the notation $\omega/df$ to denote the  
$(n-1)$-form $\psi$. Notice that the existence of $\omega/df$ is local, but its restriction on any regular fiber $f^{-1}(t)$ is unique. In special case where $f$ is non-singular on $U$ (i.e. $\nabla f(x)\ne 0, \forall x\in U$), a global expansion of $\omega/df$ can be written down in  the following form
\begin{equation}
\omega/df = g(x)\sum_{i=1}^n(-1)^{i-1}\frac{\partial_if}{\|\nabla f\|^2}dx_1\wedge\cdots\wedge \widehat{dx_i}\wedge \cdots\wedge dx_n\,,
\end{equation}
where $\widehat{}$ means to omit the corresponding factor.

The Gelfand-Leray form of the volume element and the area form on hypersurfaces have an intimate relationship. We shall clarify this relationship in a more general situation.
Consider now an $n$-dimensional  smooth manifold $M$, oriented by a volume form $\omega$ and assume $S\subset M$ is a smooth hypersurface oriented by a unit normal vector field $\mathbf{n}$. Suppose that $g$ is the Riemannian metric on $M$. We  denote by $\omega_g$ the volume element on $U$ induced by $g$. Then the hypersurface $S$ has a certain area element $\sigma_g$, which is compatible with the orientation of $S$.

Here, we recall the contraction of  forms on  finite-dimensional vector spaces. Let $V$ be a vector space of dimension $n$ and let $v$ be some fixed vector in  $V$. We define the operator $\iota_v : \Lambda^k(V^*) \rightarrow \Lambda^{k-1}(V^*)$ by
\begin{equation}\label{phepco}
\iota_{v}(\omega)(v_1,\ldots,v_{k-1})= \omega(v,v_1,\ldots,v_{k-1})
\end{equation}
where $\Lambda^k(V)$ denotes the vector space of $k-$linear forms on $V$.

The operator $\iota_v(\omega)$ is called the contraction of
forms on $V$ along vector $v$.

 With above notations, we have the following proposition whose proof can be found in \cite{Lor} (Chapter $13$).
\begin{proposition}
Let $\iota_{\mathbf{n}}$ be the contraction of differential forms along the unit normal vector field $\mathbf{n}$ of $S$. Then, we have
\begin{equation}\label{qqq} 
\sigma_g = \iota_{\mathbf{n}}\omega_g\big|_{S}
\end{equation}
\end{proposition}

Consider some open $U$ in $\mathbb{R}^n$ as  an $n$-dimensional differentiable manifold, orientated in the standard way by its tangent spaces $T_xU$. Let $f: U\to \mathbb{R}$ be a smooth function. 
It is well known that
if $t_0$ is not a critical value of $f$ then the hypersurface $S := \{x\in U: f(x)=t_0\}$ is an orientable, $(n-1)$-dimensional smooth submanifold  of $U$.
\begin{definition}
Let $S$ be a smooth hypersurface $f^{-1}(t_0)$, a vector field $v$ on $S$ points outward along $S$ if, for any $x\in S$, we have $\left\langle\nabla f(x), v(x)\right\rangle > 0$.
\end{definition} 

\begin{proposition}\label{promain}
Assume that $t_0\in \mathbb{R}$ is a regular value of $f$, hence $S := f^{-1}(t_0)$ is a smooth hypersurface of $U$. Denote by $\mathbf{n}$ the unit normal vector field pointing outward  along  $S$. 
    
Let $\omega = dx_1\wedge\cdots\wedge dx_n$ be Euclidean volume element in $\mathbb{R}^n$ and $\sigma$ be the corresponding area element on $S$. Then
\begin{equation}\label{tt-dt}
df(\mathbf{n})\frac{\omega}{df}\Big|_S=\sigma
\end{equation}
where  $\dfrac{\omega}{df}$ is the Gelfand-Leray form.
\end{proposition}
\begin{proof} Assume $(E_2,\ldots,E_n)$ is an arbitrary oriented basis  on $S$ such that $(E_1:= \mathbf{n},E_2,\ldots,E_n)$ gives the canonical orientation of the tangent fibration $TU$.

By using \eqref{qqq} and \eqref{phepco}, we have
$$\sigma(E_2,\ldots,E_n)=(i_{\mathbf{n}}\omega)(E_2,\ldots,E_n) = \omega(\mathbf{n},E_2,\ldots,E_n)$$
From the above definitions, we can write
$$\begin{array}{rcl}\medskip
\omega(\mathbf{n},E_2,\ldots,E_n) & = &(df\wedge \dfrac{\omega}{df})(E_1,E_2,\ldots,E_n)  \\ \medskip
 & = & \sum_{i=1}^n(-1)^{i-1}df(E_i)\dfrac{\omega}{df}(E_1,\ldots,\widehat{E_i},\ldots,E_n) \\
 & = & df(\mathbf{n})\dfrac{\omega}{df}(E_2,\ldots,E_n)\qquad (\text{since }
df(E_i)=0,\forall i\ge 2)
\end{array}$$
This yields the equality.\qed
\end{proof}

\section{Volume of sublevel sets and area of level sets}
We begin this section by fixing some notations. Let $f:\mathbb{R}^n\rightarrow \mathbb{R}$ be a given function. For each $t\in \mathbb{R}$, we denote by $E_t$ the sublevel set of $f$ at the level $t$,
\begin{equation}
E_t =  \{x\in \mathbb{R}^n : f(x) \le t\}
\end{equation}
and by $V(t)$ the $n-$dimensional Lebesgue measure of $E_t$.  Notice that $V(t)$ may not  exist except for cases where $f$ is  bounded below, continuous and proper in the sense that the preimage of a compact set under $f$ is again compact.

In relating to $V(t)$, we are interested in the $(n-1)-$dimensional Lebesgue measure of the level set $\{f=t\}$, which will be denoted by $A(t)$. For the sake of brevity, we call them,  correspondingly, the volume and the area.

For a given smooth $n-$form $\omega$ on $\mathbb{R}^n$ such that $\text{supp} \omega \cap \{f=t\}$ is compact for almost $t\in \mathbb{R}$, we are able to consider the real function
\begin{equation}\label{hamGJ}
J(t) := \int_{f=t}\dfrac{\omega}{df}\quad,
\end{equation}
which is called the Gelfand-Leray function generated by $\omega$.
This function is well-defined at regular values of $f$. If $f$ is sufficiently smooth, $J(t)$ is differentiable almost everywhere, except for the set of critical values of $f$. Since the set of critical values of a smooth function has measure zero, we can ignore these values in the context of Lesbegue measure. Thus we have
\begin{proposition}\label{fubini} Assume that the support of $n-$form $\omega$ is disjoint with the critical set  of a smooth function $f$. Then we have
\begin{equation}
\int_{\mathbb{R}^n}\omega = \int_{-\infty}^{+\infty}\left(\int_{f=t}\frac{\omega}{df}\right)dt
\end{equation}
\end{proposition}
\begin{proof}
 From the definition of Gelfand-Leray form, we have
$$\int_{\mathbb{R}^n}\omega=\int_{\mathbb{R}^n}df\wedge\dfrac{\omega}{df}$$
For the integral on the right hand side, by changing of variable
$$(x_1,\ldots,x_n)\mapsto (t:=f,x_2,\ldots,x_n)$$
(assuming $\partial f/\partial x_1\ne 0$) and by using the Fubini formula, we obtain
$$\int_{\mathbb{R}^n}df\wedge\dfrac{\omega}{df}= \int_{-\infty}^{+\infty}\left(\int_{f=t}\frac{\omega}{df}\right)dt$$
Proposition \ref{fubini} is complete.\qed
\end{proof}
Recall that $V(t)$ is the volume of $E_t = \{f\le t\}$. So
$$V(t)= \int_{E_t}\omega$$
where $\omega  = dx_1\wedge\cdots\wedge dx_n$ is the standard volume element in the Euclidean space $\mathbb{R}^n$.
\begin{proposition} Let $f:\mathbb{R}^n\rightarrow \mathbb{R}$ be a differentiable function having an isolated minimum at $0$, with $f(0)=0$. Moreover, assume that $f$ is proper and $df\ne 0$ on each  fiber $f^{-1}(t)$, for all $t>0$. 
Then
\begin{equation}\label{daohamthetich}
\frac{d}{dt}V(t) = J(t),\quad\forall t>0.
\end{equation}
\end{proposition}
\begin{proof} It is obvious that $V(t)$ is a non-decreasing function. 
By making use of Proposition \ref{fubini}, we can write
\begin{eqnarray*}
V(t) &=& \int_{0}^t\left(\int_{f=s}\frac{dx}{df}\right)ds=\int_{0}^tJ(s)ds\\
\end{eqnarray*}
This yields the equality \eqref{daohamthetich}.\qed
\end{proof}

One difficult problem is to find conditions  under which  $V(t)$ is finite for any~$t$, as well as to estimate the asymptotic behavior of $V(t)$ when $t$ tends to some critical value $t_0$. On the other hand, the normal estimation of $V(t)$ can imply the asymptotic behavior of oscillatory integrals \eqref{tpdd} and vice versa (see \cite{Arn}).
These two parallel problems are known as a multidimensional version of van der Corput lemma, which shows to be a powerful tool in many contexts  (see \cite{pss,ccw,cw}).
\begin{theorem}\label{dlchinh}
Assume as above that $f:\mathbb{R}^n\rightarrow \mathbb{R}$ is differentiable, having an isolated minimum at the origin $0$ such that $f(0)=0$. Moreover assume that $f$ is proper and $df$ is $\neq 0$ on each fiber $f^{-1}(t)$   for all $t>0$ and the hypersurface $f^{-1}(t)$ are connected  for all $t\in (0,\epsilon)$.
Then, for a fixed $t\in (0,\epsilon)$, there exists a $\xi\in f^{-1}(t)$ such that
\begin{equation}\label{dlmain}
V'(t) = \frac{A(t)}{\|\nabla f(\xi)\|}
\end{equation}
\end{theorem}
\begin{proof}
 By virtue of Proposition~\ref{promain}, we have
\begin{eqnarray*}
V'(t) &=& \int_{f=t}\frac{\omega}{df}\\
&=&\int_{f=t}\dfrac{\sigma}{df(\mathbf{n})}\\
\end{eqnarray*}
where $\sigma$ still denotes the area element of $\{f=t\}$ induced by the Euclidean volume element in $\mathbb{R}^n$.

By restricting on the hypersurface $f=t$, we have
$$df(\mathbf{n})=\langle\nabla f,\frac{\nabla f}{\|\nabla f\|}\rangle=\|\nabla f\|$$
The connectedness and compactness of $f^{-1}(t)$ enable us to infer that there exists a $\xi=\xi_t\in f^{-1}(t)$ such that 
\begin{align*}
\int_{f=t}\dfrac{\sigma}{df(\mathbf{n})}&=\int_{f=t}\dfrac{\sigma}{\|\nabla f(x)\|}\\
&=\frac{1}{\|\nabla f(\xi)\|}\int_{f=t}\sigma=\frac{A(t)}{\|\nabla f(\xi)\|}
\end{align*}
This yields the formula \eqref{dlmain}.\qed
\end{proof}
\begin{remark}
For $f\in C^1(\mathbb{R}^n)$ such that its fibers $f^{-1}(t)$ are compact connected non-singular hypersurfaces, the quantity $\|\nabla f(\xi)\|^{-1}$ in \eqref{dlmain} is nothing but the integral mean value of function $\frac{1}{\|\nabla f(x)\|}$  on the fiber $f^{-1}(t)$. As a corollary of Theorem~\ref{dlchinh} one obtains the known result:
\end{remark}
\begin{corollary}
The derivative of the volume of a ball with respect to the radius equals
the area of its boundary.
\end{corollary}
\begin{proof}
 Consider a closed ball of radius $t > 0$ centered at the origin in $\mathbb{R}^n$. It is merely the sublevel set of function $f(x) = \sqrt{x_1^2+\cdots+x^2_n}$ at the level $t$.
 
The assertion now comes from \eqref{dlmain} with notice that $\|\nabla f(\xi)\|
= 1$.\qed
\end{proof}
Hereafter, we also raise an estimation of $V(t)$. A further discussion about this is mentioned in Section~\ref{kl}.
\begin{corollary}
Assume in addition to the hypotheses of Theorem~\ref{dlchinh} that $f$ is an analytic function in a neighborhood of $0\in \mathbb{R}^n$ such that $f(0)=0$, then we have the inequality
\begin{equation}\label{bdt}
V'(t)\le \frac{A(t)}{t^{\nu}},
\end{equation}
with some $0<\nu<1$.
\end{corollary}
\begin{proof} It is sufficient to apply the gradient inequality (see \cite{loj}) to \eqref{dlmain}. \qed
\end{proof}
\section{Extension: Case of piece-wise smooth functions} We next extend the above result to a larger class of piece-wise smooth functions on $\mathbb{R}^n$. In other words, the fibers of $f$ do not need to be smooth, provided that they admit a decomposition into a finite number of smooth connected components (Fig.~\ref{hinh1}). This seems be appropriate for our considerations because the volume is not modified up to null-measure sets.
\begin{figure}[h]
\centering
\subfigure[\label{hinh1}]
  {\includegraphics[bb = 0 0 331 333, scale=0.3]{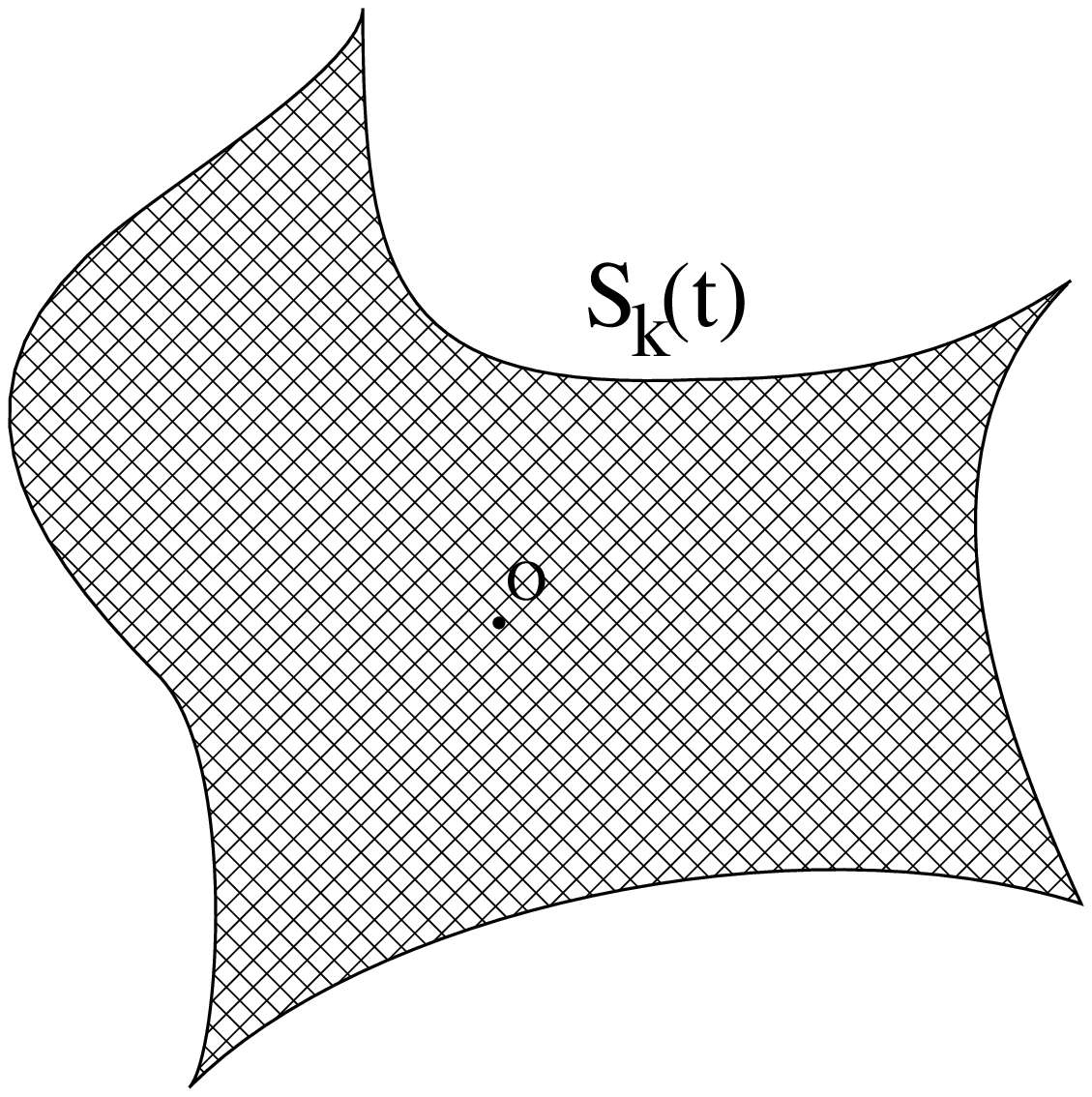}}\quad
\subfigure[\label{hinh2}]
  {\includegraphics[width=.45\linewidth]{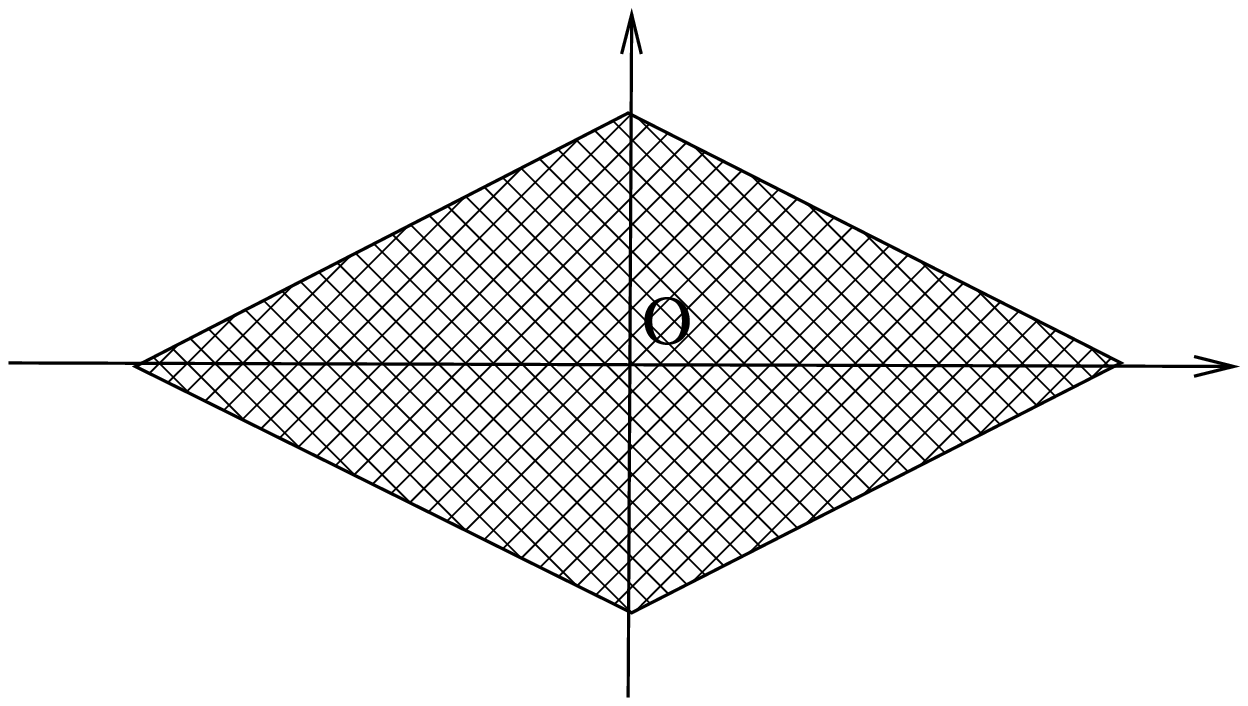}}\hfill
\caption{Sublevel sets}
\end{figure}
\begin{theorem}\label{morong}
Let $f:\mathbb{R}^n\rightarrow \mathbb{R}$ be a non-negative, piece-wise smooth function satisfying the following conditions:
\begin{enumerate}
\item[i)] For each $t\in (0,\epsilon)$, the sublevel set $\{f\le t\}$ is a compact neighborhood of the origin.
\item[ii)] There is a positive integer $m$ such that each level set $\{f=t\}$ is a discrete union of $m$ connected smooth hypersurfaces $S_k(t)$ $(k=0,\ldots,m)$, on which $\|\nabla f\|\ne 0$.
\end{enumerate}
Then there exists $\xi_k=\xi_k(t)\in S_k(t)$ satisfying the following
\begin{equation}\label{moronga}
V'(t) = \sum_{k=1}^m\frac{A_k(t)}{\|\nabla f(\xi_k)\|}
\end{equation}
where $A_k(t)$ denotes the area of $S_k(t)$ for $k=1,2,...,m$.
\end{theorem}
\begin{proof} By hypothesis and Proposition~\ref{promain}, we can write
\begin{eqnarray*}
V'(t) &=& \int_{f=t}\frac{\omega}{df}=\int_{f=t}\dfrac{\sigma}{df(\mathbf{n})}\\
&=&\sum_{k=1}^m\int_{S_k(t)}\dfrac{\sigma}{df(\mathbf{n})}.\\
\end{eqnarray*}
The remain of the proof is similar to the proof of  Theorem~\ref{dlchinh}. \qed
\end{proof}

This extension enables us to treat a significantly larger class of piece-wise smooth functions, especially smooth functions of variables $|x_1|,\ldots,|x_n|$ on $\mathbb{R}^n$. 

Hereafter we are interested in piecewise-linear  functions
$$f(x_1,\ldots,x_n)=a_1|x_1|+a_2|x_2|+\cdots+a_n|x_n|$$
where $a_i> 0$ for all $i=1,2,\ldots,n$ and $\sqrt{a_1^2+\cdots+a_n^2}=1$.

It is obvious that the sublevel set of $f$ at some level $t>0$ is a convex (compact) polyhedron whose faces have the same distance $\frac{t}{\|a\|}$ to the origin (Fig.~\ref{hinh2}). In other hand, $\|\nabla f(x)\|=\|a\|$ for all $x$ belongs to the interior of octans of $\mathbb{R}^n$. Thus, \eqref{moronga} now reads 
$$V'(t)=\sum_{k=1}^{2^n}\frac{A_k(t)}{\|a\|}=\frac{A(t)}{\|a\|}$$

By the same argument as above we can re-establish some basic events concerning the relationship between volume and area of a polyhedron. Let us consider a polyhedron $\mathcal{P}$ in $\mathbb{R}^n$ having the origin as its interior point. We suppose that $\mathcal{P}$ is \textit{regular}, in the sense that the origin is equidistant from  its faces. 

 Assume without loss of generality that this distance is unit. Denote by $t\mathcal{P}$ the dilation  of $\mathcal{P}$ of coefficient $t>0$. Then $t\mathcal{P}$ can be considered as sublevel sets of a certain piecewise-linear function $F(x)$, whose restriction to the $i^{th}$-face of $t\mathcal{P}$ is a linear functional of the form $F_i(x)=<a_i,x>$, where $\|a_i\|=1$.
 
 This leads the fact that the derivative of the volume $V(t)$ of $t\mathcal{P}$ with respect to $t$ equals the area of its boundary.
\section{Closing remarks}\label{kl}
Theorem~\ref{dlchinh} gives an intrinsic relation between geometrical measures of a certain (smooth) function. It may be useful in estimating  these measures via the gradient of $f$.

For tame objects (e.g. see \cite{yomdin}), the area $A(t)$ must be small and tend to $0$, whenever $t$ tends to $0$. Hence, we can upper bound it by some constant  for all $t\in (0,\epsilon)$. Moreover, in case $f$ has an absolute minimum value $f(0)=0$,  we can  deduce from \eqref{bdt} that
\begin{equation}\label{danhgia}
V(t)= \int_0^t V'(s)ds\le Ct^{1-\nu},
\end{equation}
where $\nu\in (0,1)$.

This can give a sharp estimation of the volume of sublevel sets as long as we accurately compute the exponent $\nu$ in the gradient inequality.
The inequality in \eqref{danhgia} is also concerned with estimating the decay rate of oscillatory and exponential integrals (see \cite{pss,ccw,vas}).
%

\section*{Acknowledgments}  We would like to thank the reviewers for their helpful suggestions and comments on the paper. This research is funded
 by Vietnam National Foundation for Science and Technology Development (NAFOSTED) under grant number $101.04-2017.324$.


\begin{thebibliography}{}
\bibitem{Leray}J.~Leray, Le calcul diff\'{e}rentielle et int\'{e}grale sur une vari\'{e}t\'{e} analytique complexe. { Bull. Soc. Math. France, 87 (1959) pp. 81–-180.}


\bibitem{Arn} V.I. Arnold, A.N. Varchenko, S.M. Gusein-Zade, { Singularities of Differentiable Maps}: Vol. 2: Monodromy and asymptotic integrals (Monographs in Mathematics), Birkh\"{a}user Boston (1988).


\bibitem{Lor} Loring W.Tu, { An Introduction to Manifolds. \/}  (Universitext), Springer New York (2008).

\bibitem{pss} D.H. Phong,  E.M. Stein and J. Sturm,  Multilinear level set operators, oscillatory integral operators, and Newton polyhedra. { Mathematische Annalen 319.3 (2001): 573-596.}

\bibitem{ccw} A. Carbery, M. Christ and J. Wright,  Multidimensional van der Corput and sublevel set estimates. { Journal of the American Mathematical Society 12.4 (1999): 981-1015.}

\bibitem{cw} A. Carbery and J. Wright. What is van der Corput's lemma in higher dimensions?. { Publicacions Matem\`{a}tiques (2002): 13-26.}

\bibitem{loj} S. {\L}ojasiewicz, Une propri\'{e}t\'{e} topologique des sous-ensembles analytiques r\'{e}els, Colloques
internationaux du C.N.R.S.: {Les \'{e}quations aux d\'{e}riv\'{e}es partielles, Paris (1962), Editions du
C.N.R.S., Paris, 1963, pp. 87-89.}

\bibitem{stein} E.M. Stein, { Harmonic Analysis: Real-Variable Methods, Orthogonality, and Oscillatory Integrals.}(Vol. 43. Princeton University Press, 2016).

\bibitem{yomdin} Y.Yomdin, and Georges Comte. { Tame geometry with application in smooth analysis}. (Springer, 2004).

\bibitem{vas} V.A.Vassiliev,  Asymptotic exponential integrals, Newton's diagram, and the classification of minimal points. { Funktsional'nyi Analiz i ego Prilozheniya 11.3 (1977): 1-11.}
\end{thebibliography}
\end{document}